\documentclass[11pt]{article}
\usepackage{amsmath, amssymb, amscd, amsthm, amsfonts}
\usepackage{graphicx}
\usepackage{hyperref}
\usepackage{multirow}
\usepackage{float}
\usepackage{tabularx,booktabs,caption,ragged2e}
\newcolumntype{L}{>{\RaggedRight\arraybackslash}X}
\usepackage[export]{adjustbox}

\title{On the dimensions of certain spaces of vector-valued cusp forms}
\author{
    Darshan Nasit\\
    Department of Mathematics\\
    Indian Institute of Science Education and Research Pune\\
    Pune, India. 411008 \\
  \texttt{nasit.darshan@students.iiserpune.ac.in} }
\date{}

\newtheorem{theo}{{\bf{Theorem}}}
\newtheorem{prop}[theo]{{\bf Proposition}}
\newtheorem{lem}[theo]{{\bf Lemma}}
\newtheorem{coro}[theo]{{\bf Corollary}}
\newtheorem{rmk}{{\bf{Remark}}}

\usepackage{blindtext}
\usepackage{titlesec}

\begin{document}

\maketitle

\begin{abstract}
    Given an irreducible representation of $SL_2(\mathbb{F}_q)$ for an odd prime $q\geq5$, we find the dimension of the space of cusp forms with respect to the full modular group taking values in the representation space. The dimension equals the multiplicity of the representation in the space of classical cusp forms with respect to the principal congruence subgroup of level $q$.
\end{abstract}


\section{Introduction}

Let $\mathcal{H}$ be the complex upper half plane and $\Gamma$ be a finite index subgroup of $SL_2(\mathbb{Z})$. Let $\rho:\Gamma\rightarrow GL(V)$ be a finite-dimensional complex representation such that $\Gamma'=\ker\rho$ has a finite index in $\Gamma$. A function $f:\mathcal{H}\rightarrow V$ is holomorphic cuspidal if each coordinate $f_i:\mathcal{H}\rightarrow\mathbb{C}$ is holomorphic cuspidal (i.e., holomorphic and $\int_{0}^{1}f_i(x+\iota y)dx=0$ for any $y>0$). For an integer $k\geq 0$, the space of cusp forms of weight $k$ for a subgroup $\Gamma$ in the representation $(\rho,V)$ is 
$$\mathcal{S}_k(\Gamma,\rho)=\big\{f:\mathcal{H} \rightarrow V  \text{ holomorphic cuspidal such that } f(\gamma z)=j(\gamma,z)^k\rho(\gamma)(f(z)),\textbf{ }
 \forall\gamma\in\Gamma\big\},$$ where $j$ is the usual automorphy factor.
See Chapter 8 of \cite{SHI} for more details.

It is interesting to find the dimension of $\mathcal{S}_k(\Gamma,\rho)$ since the space of cusp forms (holomorphic and antiholomorphic) $\mathcal{S}_k(\Gamma,\rho)\oplus \overline{\mathcal{S}_k(\Gamma,\rho)}$ is isomorphic to the parabolic cohomology $H^1_P(\Gamma, \rho\otimes\operatorname{Sym}^{k+2}(\mathbb{C}^2))$, via the Eichler-Shimura Isomorphism (Theorem 8.4 in \cite{SHI}).

Fix an odd prime $q\geq5$ and let $\mathbb{F}:=\mathbb{Z}/q\mathbb{Z}$ denote the finite field of order $q$. We have an irreducible representation of $SL_2(\mathbb{Z})$ naturally for an irreducible representation $(\rho, V_{\rho})$ of $SL_2(\mathbb{F})$ as the composition $$SL_2(\mathbb{Z})\rightarrow SL_2(\mathbb{F})\rightarrow GL(V_{\rho}).$$
We will denote the above representation of $SL_2(\mathbb{Z})$ also by $(\rho,V_{\rho})$. Since $SL_2(\mathbb{F})$ is a finite group, the kernel of this representation has a finite index in $SL_2(\mathbb{Z})$.

Any irreducible representation of $SL_2(\mathbb{F})$ is either a subrepresentation of a principal series representation or a cuspidal representation. A principal series representation $\rho(\alpha)$ is characterized by a character $\alpha$ of $\mathbb{F}^*$. It is an irreducible representation of $SL_2(\mathbb{F})$ unless $\alpha$ is either a trivial or quadratic character. If $\alpha={\bf{1}}$ is the trivial then $\rho({\bf{1}})$ decomposes into the direct sum of the trivial representation and the Steinberg representation. If $\alpha=\zeta_e$ is a quadratic character, then $\rho(\zeta_e)=\rho(\zeta_e)^+\oplus \rho(\zeta_e)^-;$ as decomposition into inequivalent irreducible representations. Let $\mathbb{E}$ be a quadratic extension of $\mathbb{F}$ and let $\mathbb{E}_1^*$ be the subgroup of $\mathbb{E}^*$ consisting of norm $1$ elements. For a nontrivial character $\tau$ of $\mathbb{E}_1^*$, one can attach a cuspidal representation $\omega(\tau)$. A cuspidal representation $\omega(\tau)$ is irreducible unless $\tau$ is a quadratic character. Similar to the case of a principal series representation, if $\tau=\zeta_0$ is a quadratic character, then $\omega(\zeta_0)=\omega(\zeta_0)^+\oplus \omega(\zeta_0)^-;$ as decomposition into inequivalent irreducible representations.  We can view $\mathbb{E}_1^*$ as a subgroup of $SL_2(\mathbb{F})$ and let $\Gamma_{\mathbb{E}}(q)\subset SL_2(\mathbb{Z})$ be the inverse image of $\mathbb{E}_1^*\subset SL_2(\mathbb{F})$. Let $\overline{\alpha}$ denote the complex conjugation of a character $\alpha$. Let $\Gamma(q)$ be the principal congruence subgroup of level $q$, $\Gamma_1(q)$ and $\Gamma_0(q)$ be the standard congruence subgroups. Let $\psi_n$ denote the character on $\mathbb{{F}} $ defined as $x\mapsto e^{2\pi\iota \frac{nx}{q}}$ for any $n\in \mathbb{F}$. The main results of this paper are as follows.

\begin{theo}\label{theo1}
    For the principal series representation $\rho(\alpha)$, $$\dim \mathcal{S}_k(SL_2(\mathbb{Z}),\rho(\alpha)) = \begin{cases}\dim \mathcal{S}_k(\Gamma_0(q),\alpha) & \alpha(-1)=(-1)^k \\ 0 & \alpha(-1)\neq (-1)^k.\end{cases}$$
\end{theo}

\begin{coro}\label{coro2}
    For the Steinberg representation, $$\dim \mathcal{S}_k(SL_2(\mathbb{Z}),{\bf{St}}) = \begin{cases}\dim \mathcal{S}_k(\Gamma_0(q))-\dim \mathcal{S}_k(SL_2(\mathbb{Z})) & k \text{ is even}, \\ 0 & k \text{ is odd}.\end{cases}$$
\end{coro}

\begin{prop}\label{prop3}
    For the quadratic character $\zeta_e$ on $\mathbb{F}^*$, if $\zeta_e(-1)\neq (-1)^k$ then $\mathcal{S}_k(SL_2(\mathbb{Z}),\rho(\zeta_e)^{\pm})=0$, and if $\zeta_e(-1)= (-1)^k$ then $$\dim \mathcal{S}_k(SL_2(\mathbb{Z}),\rho(\zeta_e)^{\pm}) =  \textbf{ } \frac{1}{2} \dim \mathcal{S}_k(\Gamma_0(q),\zeta_e)\pm \frac{\zeta_e(-1)}{(q-1)}\sum_{n\in\mathbb{F}^*}\zeta_e(n) \dim \mathcal{S}_k(\Gamma_1(q),\psi_n).$$
\end{prop}

\begin{theo}\label{theo4}
    For cuspidal representation $\omega(\tau)$, if $\tau(-1)\neq (-1)^k$ then $\mathcal{S}_k(SL_2(\mathbb{Z}),\omega(\tau))=0$, and if $\tau(-1)=(-1)^k$ then $$\dim \mathcal{S}_k(SL_2(\mathbb{Z}),\omega(\tau)) = \frac{k-1}{12}(q^2-1)-\frac{1}{2}(q-1) - \frac{1}{2}\dim \mathcal{S}_k(\Gamma_{\mathbb{E}}(q),\overline{\tau}) - \frac{1}{2}\dim \mathcal{S}_k(\Gamma_{\mathbb{E}}(q),\tau).$$
\end{theo}

\begin{prop}\label{prop5}
    If $\zeta_0(-1)\neq(-1)^k$ then $\mathcal{S}_k(SL_2(\mathbb{Z}),\omega(\zeta_0)^{\pm})=0$, and if $\zeta_0(-1)=(-1)^k$ then 
    \begin{align*}
        \dim \mathcal{S}_k(SL_2(\mathbb{Z}),\omega(\zeta_0)^{\pm}) = & \textbf{ } \frac{k-1}{24}(q^2-1)-\frac{1}{4}(q-1)  - \frac{1}{2}\dim \mathcal{S}_k(\Gamma_{\mathbb{E}}(q),\zeta_0) \\ & \textbf{ } \textbf{ } \pm \frac{\zeta_e(-1)}{q-1}\sum_{n\in\mathbb{{F}^*}} \zeta_e(n)\mathcal{S}_k(\Gamma_1(q),\psi_n).
    \end{align*}    
\end{prop}

We prove these results by using the character of a certain representation of $SL_2(\mathbb{F})$ in the space of cusp forms $\mathcal{S}_k(\Gamma(q))$. Since $\Gamma(q)$ is a normal subgroup of $\Gamma(1)=SL_2(\mathbb{Z})$, there is a natural homomorphism $$\sigma: SL_2(\mathbb{Z})\rightarrow GL(\mathcal{S}_k(\Gamma(q))),$$ defined as $[\sigma(g)f](z)= j(g^{-1},z)^{-k}f(g^{-1}z)$. Since $\sigma$ is a representation with $\Gamma(q)\subset \ker(\sigma)$, it  factors through a representation of $SL_2(\mathbb{F})$. Suppose $\sigma$ decomposes as, $$\sigma=\bigoplus m_{\rho}\rho,$$ where the direct sum runs over the inequivalent classes of irreducible representations of $SL_2(\mathbb{F})$, and $m_{\rho}$ is the multiplicity of the irreducible representation $\rho$ in $\sigma$. We will show that the multiplicities are also related to the dimensions of the spaces of cusp forms (see Lemma \ref{lem7}).

In Section 2, we will relate the dimension of the space of cusp forms and the multiplicity of representation in the decomposition of $\sigma$. We will recall the classification of representations and the character table of $SL_2(\mathbb{F})$ in Section 3. In Section 4, we will compute the character of $\sigma$. In Sections 5 and 6, we will prove our main results using the character of $\sigma$.

\begin{rmk}
    If $\rho={\bf{1}}$ is the trivial representation then dimensions of $\mathcal{S}_k(\Gamma,{\bf{1}})=\mathcal{S}_k(\Gamma)$ have been calculated by using the Riemann-Roch theorem (\cite{SHI} and \cite{diamond}). When $\Gamma=\Gamma_0(N)$ and $\rho:\Gamma_0(N)/\Gamma_1(N)\rightarrow \mathbb{C}^*$ is a character then the dimension formulas for $\mathcal{S}_k(\Gamma,\rho)$ were given by H.\,Cohen and J.\,Oesterl\'{e} in 1977 (\cite{cohen}) and Jordi Quer in 2010 (\cite{Quer}).
    To the best of my knowledge, the explicit formulas for $\dim \mathcal{S}_k(\Gamma_1(q),\psi_n)$ and $\dim \mathcal{S}_k(\Gamma_{\mathbb{E}}(q),\tau)$ are not well-known. The future work on this project aims to find the explicit formulas for $\dim\mathcal{S}_k(\Gamma,\chi)$ for a congruence subgroup $\Gamma$ and a character $\chi:\Gamma\rightarrow\mathbb{C}^*$.
\end{rmk}

\section*{Acknowledgement}
\small{The author sincerely thanks A.\,Raghuram for suggesting the problem and engaging in helpful conversations. The author is grateful to Chandrasheel Bhagwat for his suggestions on the manuscript. He is thankful to the Institute for Advanced Study, Princeton, for a summer collaborator's grant in 2023; and to other members of this summer collaboration: A.\,Raghuram, Baskar Balasubramanyam, Chandrasheel Bhagwat, Freydoon Shahidi, and P. Narayanan. The author acknowledges support from the Prime Minister's Research Fellowship (PMRF) at IISER Pune.}

\section{Dimension and Multiplicity}\label{sec2}
Let $\mathcal{S}_k(\Gamma(q),\rho)$ be the space of classical cusp forms taking values in $V_{\rho}$ and let $$\theta:SL_2(\mathbb{Z})\rightarrow GL(\mathcal{S}_k(\Gamma(q),\rho))$$ be a representation of $SL_2(\mathbb{Z})$ defined as $[\theta(g)f](z)= j(g^{-1},z)^{-k}\rho(g)\cdot f(g^{-1}z)$.

\begin{lem}\label{lem6}
     $\mathcal{S}_k(\Gamma(q),\rho)$ is equivalent, as a representation of $SL_2(\mathbb{Z})$, to the tensor product of  $\mathcal{S}_k(\Gamma(q))\otimes V_{\rho}$, i.e., $\theta\sim \sigma\otimes \rho$. Moreover, $\chi_{\theta}=\chi_{\sigma}\cdot\chi_{\rho}$, where $\chi_{\theta},\chi_{\sigma}$ and $\chi_{\rho}$ are the characters of $\theta$, $\sigma$ and $\rho$, respectively.
\end{lem}
\begin{proof}
    Define a linear map $T:\mathcal{S}_k(\Gamma(q))\otimes V_{\rho}\rightarrow \mathcal{S}_k(\Gamma(q),\rho)$ as $[T(f\otimes v)](z)=f(z)v$. The map is well-defined since $\Gamma(q)\subset \ker(\rho)$. Let $\{e_1,\ldots, e_n\}$ is a basis of $V_{\rho}$ then for any $f\in \mathcal{S}_k(\Gamma(q),\rho)$, there exist holomorphic functions $f_1,\ldots, f_n:\mathcal{H}\rightarrow \mathbb{C}$ such that $f(z)=f_1(z)e_1+\cdots +f_n(z)e_n$. Since $\Gamma(q)\subset \ker(\rho)$, $f_i\in \mathcal{S}_k(\Gamma(q))$. Then $$T(f_1\otimes e_1+ \cdots+ f_n\otimes e_n)=f.$$ Hence, $T$ is a surjective map. Since $\dim \mathcal{S}_k(\Gamma(q),\rho) =\dim \mathcal{S}_k(\Gamma(q)) \times \dim V_{\rho}$, the linear map $T$ is an isomorphism. The lemma follows since for $f\in \mathcal{S}_k(\Gamma(q))$, $v\in V_{\rho}$ and $g\in SL_2(\mathbb{Z})$, we have
    \begin{align*}
        [\theta(g)\circ T (f\otimes v)](z) & = j(g^{-1},z)^{-k}\rho(g)\cdot [T(f\otimes v)](g^{-1}z)\\ & = j(g^{-1},z)^{-k}\rho(g) \cdot [f(g^{-1}z)v] \\ & = [\sigma(g)f](z)\rho(g)\cdot v \\ & = [T\circ \sigma\otimes \rho(g) (f\otimes v)](z).
    \end{align*}
\end{proof}
\begin{lem}\label{lem7}
    The dimension of $\mathcal{S}_k(SL_2(\mathbb{Z}),\rho)$ is equal to the multiplicity of the complex conjugation $\overline{\rho}$ of the representation $\rho$ in $\sigma$, i.e., $\dim \mathcal{S}_k(SL_2(\mathbb{Z}),\rho)=m_{\overline{\rho}}$.
\end{lem}
\begin{proof}
    Since $\mathcal{S}_k(SL_2(\mathbb{Z}),\rho) = \big\{f\in \mathcal{S}_k(\Gamma(q), V_{\rho}): \theta(g)f=f, \text{  }\forall g\in SL_2(\mathbb{Z}) \big\}$, the dimension is equal to the multiplicity of the trivial representation in $\theta$. Then $$\dim \mathcal{S}_k(SL_2(\mathbb{Z}),V_{\rho}) = \langle \chi_{\theta},{\bf{1}}\rangle = \langle \chi_{\sigma}\chi_{\rho},{\bf{1}}\rangle = \langle \chi_{\sigma},\chi_{\overline{\rho}}\rangle=m_{\overline{\rho}}.$$
    Here, second equality follows from Lemma \ref{lem6}.   
\end{proof}

\section{The Character Table}\label{sec3}

This section is to recall the classification of irreducible representations of $SL_2(\mathbb{F})$ based on Section 4.1 of \cite{Bump}.\par
Let $\mathbb{E}$ be the quadratic extension of $\mathbb{F}$, then $\mathbb{E}=\mathbb{F}(\sqrt{\Delta})$ for $\Delta\notin \mathbb{F}^{*2}$. We choose $\Delta\in\mathbb{F}$ such that $\Delta$ generates $\mathbb{F}^*$. Let $z=x+\sqrt{\Delta} y \in\mathbb{E}$ then define $\overline{z}:=x-\sqrt{\Delta}y$. The norm $N: \mathbb{E}^*\rightarrow\mathbb{F}^*$ is defined by $N(z)=z\overline{z}\in \mathbb{F}^*$. Let $\mathbb{E}_1^*$ denote the kernel of the norm map. \par
\begin{rmk}\label{rmk2}
    We can view $\mathbb{E}^*_1$ as a subgroup of $SL_2(\mathbb{F})$ via the embedding $z=x+\sqrt{\Delta} y\mapsto \begin{pmatrix}
        x & \Delta y\\ y & x
    \end{pmatrix}$.
\end{rmk}
Let $B=\bigg\{\begin{pmatrix} a & b \\ 0 & a^{-1} \end{pmatrix}\in SL_2(\mathbb{F})\bigg\}$, $U=\bigg\{\begin{pmatrix} 1 & b \\ 0 & 1 \end{pmatrix}: b\in \mathbb{F}\bigg\}\simeq \mathbb{F}$ and $T=\bigg\{\begin{pmatrix} a & 0 \\ 0 & a^{-1} \end{pmatrix}: a\in \mathbb{F}^*\bigg\}\simeq \mathbb{F}^*$ be the Borel subgroup, its unipotent radical and the diagonal torus, respectively. A character $\alpha: T\cong \mathbb{F}^*\rightarrow\mathbb{C}^*$ can be extended to the character of $B$, whose restriction on $U$ is ${\bf{1}}$. Let $\rho(\alpha)$ denote the induced representation of $B$ to $SL_2(\mathbb{F})$, which is called principal series representation.
\begin{itemize}
    \item If $\alpha={\bf{1}}$ is the trivial character then $\rho({\bf{1}})={\bf{1}}\oplus \bf{St}$, where ${\bf{1}}$ is a trivial representation, and $\bf{St}$ is the Steinberg representation of dimension $q$.
    \item If $\alpha$ is a quadratic character, denoted by $\zeta_e$, then $\rho(\zeta_e)=\rho(\zeta_e)^+\oplus \rho(\zeta_e)^-$, where $\rho(\zeta_e)^{\pm}$ are irreducible representations of dimension $\dfrac{q+1}{2}$. Note that, a non-trivial quadratic character exists and it is unique. Moreover, $\zeta_e(x)=1$ for all $x\in\mathbb{F}^{*2}$ and $\zeta_e(x)=-1$ for all $x\in\mathbb{F}^*\backslash \mathbb{F}^{*2}$.
    \item If $\alpha$ is neither a trivial nor a quadratic character then $\rho(\alpha)$ is an irreducible representation of dimension $q+1$. Moreover, $\rho(\alpha)\sim \rho(\overline{\alpha}) \sim \overline{\rho(\alpha)}$. 
\end{itemize}

Let $\tau:\mathbb{E}^*\rightarrow \mathbb{C}^*$ be a character such that $\tau|_{\mathbb{E}_1^*}$ is a non-trivial character and let $$V_{\tau}=\big\{f:\mathbb{E}\rightarrow\mathbb{C}:\textbf{ }f(zx)=\tau(z)^{-1}f(x), \forall z\in \mathbb{E}^*_1\big\}.$$ The space $V_{\tau}$ is preserved under the special Weil representation (Section 4.1 in \cite{Bump}). The representation $(\omega(\tau),V_{\tau})$ is called a cuspidal representation for a character $\tau$.
\begin{itemize}
    \item If $\tau$ is a quadratic character, denoted by $\zeta_0$ then $\omega(\zeta_0)=\omega(\zeta_0)^+\oplus \omega(\zeta_0)^-$, where $\omega(\zeta_0)^{\pm}$ are irreducible representations of dimension $\dfrac{q-1}{2}$.
    \item If $\tau$ is a nontrivial character such that $\tau\neq \zeta_0$ then $\omega(\tau)$ is an irreducible representation of dimension $q-1$. Moreover, $\omega(\tau)\sim \omega(\overline{\tau})$.
\end{itemize}

Then the character table for $SL_2(\mathbb{F})$ is 

\begin{table}[ht]
    \centering
    \begin{tabular}{|c|c|c|c|c|c|c|c|}
         \hline
		&  & \multirow{-1.5}{3em}{Class Repre.} & $\begin{pmatrix}1 & 0 \\ 0 & 1\end{pmatrix}$ & $\begin{pmatrix}-1 & 0 \\ 0 & -1\end{pmatrix}$ & $(-1)^{\eta}\begin{pmatrix}1 & \gamma \\ 0 & 1 \end{pmatrix}$ & $\begin{pmatrix}x & \Delta y \\ y & x\end{pmatrix}$ & $\begin{pmatrix} x & 0\\ 0 & x^{-1} \end{pmatrix}$\\ 
	   \hline
	    &  & Sizes & $1$ & $1$ & $\frac{q^2-1}{2}$ & $q(q-1)$ & $q(q+1)$ \\
	     \hline 
	    &  &  & 1 & 1 & 4 & $\frac{q-1}{2}$ & $\frac{q-3}{2}$ \\
	   \hline
        Repr. & Dim. &  &  &  & $\substack{\eta\in\{0,1\}\\ \gamma\in\{1,\Delta\}}$ & $z\in\mathbb{E}_1^*-\{\pm1\}$ & $x\neq \pm1 \in \mathbb{F}^*$ \\ [0.5ex]
        \hline\hline
        $\rho(\alpha)$ & $q+1$ & $\frac{q-3}{2}$ & $(q+1)$ & $(q+1)\alpha(-1)$ & $\alpha(-1)^{\eta}$  & $0$ &  $\alpha(x)+\overline{\alpha}(x)$ \\
        \hline
        $\bf{St}$ & $q$ & $1$ & $q$ & $q$ & $0$ & $-1$ & $1$\\
        \hline
        ${\bf{1}}$ & 1 & $1$ & $1$ & $1$ & $1$ & $1$ & $1$ \\
        \hline
        $\omega(\tau)$ & $q-1$ & $\frac{q-1}{2}$ & $(q-1)$ & $(q-1)\tau(-1)$ & $-\tau((-1)^{\eta})$ & $-\tau(z)-\overline{\tau}(z)$ & $0$ \\
        \hline
        $\rho(\zeta_e)^{\pm}$ & $\frac{q+1}{2}$ & $2$ & $\frac{q+1}{2}$ & $\frac{q+1}{2}\zeta_e(-1)$ & $\zeta_e(-1)^{\eta}\zeta_e^{\pm}(\gamma)$ & $0$ & $\zeta_e(x)$ \\
        \hline
        $\omega(\zeta_0)^{\pm}$ & $\frac{q-1}{2}$ & $2$ & $\frac{q-1}{2}$ & $\frac{q-1}{2}\zeta_0(-1)$ & $\zeta_0(-1)^{\eta}\zeta_0^{\pm}(\gamma)$ & $-\zeta_0(z)$ & $0$\\ 
        \hline
    \end{tabular}
    \caption{Character Table}
    \label{character table}
\end{table}

where $\zeta_e^{\pm}(1)=\frac{1}{2}(1\pm\sqrt{q\epsilon})$, $\zeta_e^{\pm}(\Delta)=\frac{1}{2}(1\mp\sqrt{q\epsilon})$, $\zeta_0^{\pm}(1)=\frac{1}{2}(-1\pm\sqrt{q\epsilon})$, $\zeta_0^{\pm}(\Delta)=\frac{1}{2}(-1\mp\sqrt{q\epsilon})$ and $\epsilon=\zeta_e(-1)$. Refer to \cite{Humphreys} for more details.

\begin{rmk}\label{rmk3}
    For $x \in \mathbb{F}^*$, the element $\begin{pmatrix}1 & x \\ 0 & 1 \end{pmatrix}$ is conjugate to either $\begin{pmatrix}1 & 1 \\ 0 & 1 \end{pmatrix}$ if $x \in \mathbb{F}^{*2}$, or to $\begin{pmatrix}1 & \Delta \\ 0 & 1 \end{pmatrix}$ if $x \notin \mathbb{F}^{*2}$.
\end{rmk}

\section{The Character of The Representation on Cusp Forms}\label{sec4}

Since $\Gamma_0(q)$, $\Gamma_1(q)$ and $\Gamma_{\mathbb{E}}(q)$ are subgroups of $SL_2(\mathbb{Z})$ such that they are the inverse image under the natural map $SL_2(\mathbb{Z})\rightarrow SL_2(\mathbb{F})$ of the Borel subgroup, unipotent subgroup, and $\mathbb{E}^*_1$ respectively. Note that$\Gamma_0(q)/\Gamma(q)\cong B$, $\Gamma_1(q)/\Gamma(q)\cong U$ and $\Gamma_{\mathbb{E}}(q)/\Gamma(q)\cong \mathbb{E}_1^*$.
\begin{rmk}\label{rmk4}
    Since $\sigma(-1)f=(-1)^kf$ for all $f\in \mathcal{S}_k(\Gamma(q))$, $\chi_{\sigma}(-g)=(-1)^k\chi_{\sigma}(g)$ for all $g\in SL_2(\mathbb{Z})$.
\end{rmk}
Let $\sigma_{\mathbb{E}}$ be the restriction of $\sigma$ to $\Gamma_{\mathbb{E}}(q)$ or $\mathbb{E}^*_1$. Since $\mathbb{E}^*_1$ is an abelian group, $$\mathcal{S}_k(\Gamma(q))=\bigoplus_{\tau:\mathbb{E}^*_1\rightarrow \mathbb{C}^*}\mathcal{S}_k(\Gamma_{\mathbb{E}}(q),\tau),$$ where $\mathcal{S}_k(\Gamma_{\mathbb{E}}(q),\tau)=\big\{f\in \mathcal{S}_k(\Gamma(q)):\textbf{ }\sigma_{\mathbb{E}}(g)\cdot f=\tau(g)f,\textbf{ }\forall g\in \Gamma_{\mathbb{E}}(q)\big\}$. Since $\chi_{\sigma}(g)=\chi_{\sigma_{\mathbb{E}}}(g)$ for all $g\in \Gamma_{\mathbb{E}}(q)$, 
\begin{equation}\label{eq1}
    \chi_{\sigma}(g)=\sum_{\tau:\mathbb{E}^*_1\rightarrow \mathbb{C}^*}\tau(g) \dim \mathcal{S}_k(\Gamma_{\mathbb{E}}(q),\tau).
\end{equation}

\begin{rmk}\label{rmk5}
    We will say $\tau$ is even (or odd) character if $\tau(-1)=1$ (or $\tau(-1)=-1$ respectively). If $k$ is even (or odd) then $\mathcal{S}_k(\Gamma_{\mathbb{E}}(q),\tau)=0$ for all odd (or even) $\tau$, respectively. So, $$\sum_{\tau:\mathbb{E}^*_1\rightarrow \mathbb{C}^*}\tau(-1) \dim \mathcal{S}_k(\Gamma_{\mathbb{E}}(q),\tau)=(-1)^k \dim \mathcal{S}_k(\Gamma(q)).$$
\end{rmk}
Similarly, let $\sigma_1$ be the restriction of $\sigma$ to $\Gamma_1(q)$ or $U$. Since $U\cong \mathbb{F}$ is an abelian group, $$\mathcal{S}_k(\Gamma(q))=\bigoplus_{n\in\mathbb{F}}\mathcal{S}_k(\Gamma_1(q),\psi_n),$$ where $\mathcal{S}_k(\Gamma_1(q),\psi_n)=\big\{f\in \mathcal{S}_k(\Gamma(q)):\textbf{ }\sigma_1(g)\cdot f=\psi_n(g)f,\textbf{ }\forall g\in \Gamma_1(q)\big\}$. Since $\chi_{\sigma}(g)=\chi_{\sigma_1}(g)$ for all $g\in \Gamma_1(q)$, 
\begin{equation}\label{eq2}
    \chi_{\sigma}(g)=\sum_{n\in\mathbb{F}}\psi_n(g) \dim \mathcal{S}_k(\Gamma_1(q),\psi_n).
\end{equation}
Let $\sigma_0$ be the restriction of $\sigma$ to $\Gamma_0(q)$.
\begin{lem}\label{lem8}
    If $f\in \mathcal{S}_k(\Gamma_1(q),\psi_n)$ and $g=\begin{pmatrix} a& b\\ c&d\end{pmatrix}\in \Gamma_0(q)$ then $\sigma_0(g)\cdot f\in \mathcal{S}_k(\Gamma_1(q),\psi_{nd^2})$.
\end{lem}
\begin{proof}
    Let $\gamma=\begin{pmatrix} x& y\\ z&w\end{pmatrix}\in \Gamma_1(q)$ then $\sigma_1(\gamma)\cdot f=\psi_n(y)f=e^{2\pi\iota\frac{ny}{q}}f$. Now, $$\sigma_1(\gamma)\sigma_0(g)\cdot f= \sigma(\gamma g)\cdot f= \sigma_0(g)\sigma_1(g^{-1}\gamma g)\cdot f=\psi_n(g^{-1}\gamma g)\sigma_0(g)\cdot f=\psi_n(d^2y)\sigma_0(g)\cdot f=\psi_{nd^2}(y)\sigma_0(g)\cdot f.$$
\end{proof}
\begin{lem}\label{lem9}
    Let $g=\begin{pmatrix} a& b\\ c&d\end{pmatrix}\in \Gamma_0(q)$ such that $d\not\equiv\pm 1 (\text{mod } q)$. Then 
    \begin{equation}\label{eq3}
    \chi_{\sigma}(g)=\sum_{\phi:\mathbb{F}^*\rightarrow \mathbb{C}^*}\phi(d) \dim \mathcal{S}_k(\Gamma_0(q),\phi).
\end{equation}
\end{lem}
\begin{proof}
    Since $\mathcal{S}_k(\Gamma(q))=\bigoplus\mathcal{S}_k(\Gamma_1(q),\psi_n)$, we have a basis of $\mathcal{S}_k(\Gamma(q))$ as an union of bases of $\mathcal{S}_k(\Gamma_1(q),\psi_n)$. By using Lemma \ref{lem8}, an element of basis of $\mathcal{S}_k(\Gamma_1(q),\psi_n)$ contributes to the trace of $\sigma(g)=\sigma_0(g)$ if and only if $n=0\in\mathbb{F}$. So, $\chi_{\sigma_0}(g)=\chi_{\widetilde{\sigma_0}}(g)$ where $\widetilde{\sigma_0}:\Gamma_0(q)\rightarrow GL(\mathcal{S}_k(\Gamma_1(q)))$. Since $\Gamma_1(q)\subset \ker(\widetilde{\sigma_0})$ and $\Gamma_0(q)/\Gamma_1(q)\cong \mathbb{F}^*$ is an abelian group,$$\mathcal{S}_k(\Gamma_1(q))=\bigoplus_{\phi:\mathbb{F}^*\rightarrow \mathbb{C}^*}\mathcal{S}_k(\Gamma_0(q),\phi),$$where $\mathcal{S}_k(\Gamma_0(q),\phi)=\big\{f\in \mathcal{S}_k(\Gamma_1(q)):\textbf{ }\sigma_0(x)\cdot f=\phi(x)f,\textbf{ }\forall x\in \Gamma_1(q)\big\}$. So, the lemma follows.
\end{proof}
\begin{rmk}\label{rmk6}
    We will say $\phi$ is even (or odd) character if $\phi(-1)=1$ (or $\phi(-1)=-1$ respectively). If $k$ is even (or odd) then $\mathcal{S}_k(\Gamma_0(q),\phi)=0$ for all odd (or even) $\phi$, respectively. So, $$\sum_{\phi:\mathbb{F}^*\rightarrow \mathbb{C}^*}\phi(-1) \dim \mathcal{S}_k(\Gamma_0(q),\phi)=(-1)^k \dim \mathcal{S}_k(\Gamma_1(q)).$$
\end{rmk}
The character function of the representation $\sigma$ is:
\begin{table}[ht]
    \centering
    \begin{tabular}{|c|c|c|c|}
         \hline
        & {Class Representative} & size & $\chi_{\sigma}$ \\
    \hline
        $1$ & $\begin{pmatrix} 1& 0\\ 0&1 \end{pmatrix}$ & $1$ & $\dim \mathcal{S}_k(\Gamma(q))$ \\
    \hline
        $1$ & $\begin{pmatrix} -1& 0\\ 0&-1 \end{pmatrix}$ & $1$ & $(-1)^k\dim \mathcal{S}_k(\Gamma(q))$ \\
    \hline
        $1$ & $\begin{pmatrix} 1& 1\\ 0&1 \end{pmatrix}$ & $\dfrac{q^2-1}{2}$ & $\sum\limits_{n\in\mathbb{F}}\psi_n(1)\dim \mathcal{S}_k(\Gamma_1(q),\psi_n)$ \\
    \hline
        $1$ & $\begin{pmatrix} 1& \Delta\\ 0&1\end{pmatrix}$ & $\dfrac{q^2-1}{2}$ & $\sum\limits_{n\in\mathbb{F}}\psi_n(\Delta)\dim \mathcal{S}_k(\Gamma_1(q),\psi_n)$ \\
    \hline
        $1$ & $\begin{pmatrix} -1& -1\\ 0&-1\end{pmatrix}$ & $\dfrac{q^2-1}{2}$ & $(-1)^k\sum\limits_{n\in\mathbb{F}}\psi_n(1)\dim \mathcal{S}_k(\Gamma_1(q),\psi_n)$ \\
    \hline
        $1$ & $\begin{pmatrix} -1& -\Delta\\ 0&-1\end{pmatrix}$ & $\dfrac{q^2-1}{2}$ & $(-1)^k\sum\limits_{n\in\mathbb{F}}\psi_n(\Delta)\dim \mathcal{S}_k(\Gamma_1(q),\psi_n)$ \\
    \hline
        $\dfrac{q-3}{2}$ & $\begin{pmatrix} x^{-1}& 0\\ 0&x\end{pmatrix}$ & $q(q+1)$ & $\sum\limits_{\phi:\mathbb{F}^*\rightarrow\mathbb{C}^*}\phi(x)\dim \mathcal{S}_k(\Gamma_0(q),\phi)$ \\
    \hline
        $\dfrac{q-1}{2}$ & $\begin{pmatrix} x& \Delta y\\ y&x\end{pmatrix}$ & $q(q-1)$ & $\sum\limits_{\tau:\mathbb{E}_1^*\rightarrow \mathbb{C}^*}\tau\bigg(\begin{pmatrix} x& \Delta y\\ y&x\end{pmatrix}\bigg)\dim \mathcal{S}_k(\Gamma_{\mathbb{E}}(q),\tau)$. \\
    \hline    
    \end{tabular}
    \caption{Character of $\sigma$}
    \label{table}
\end{table}

\section{Cusp forms with values in principal series representations}

We will find the dimension of $\mathcal{S}_k(SL_2(\mathbb{Z}),\rho(\alpha_0))$, where $\rho(\alpha_0)$ is the parabolically induced representation for a character $\alpha_0:\mathbb{F}^*\rightarrow \mathbb{C}^*$.

Let $\mathcal{F}: \mathcal{S}_k(\Gamma_0(q),\alpha)\rightarrow \mathcal{S}_k(SL_2(\mathbb{Z}),\rho(\alpha))$ be a linear map defined as $[\mathcal{F}f(z)](x)=j(\Tilde{x},z)^{-k}f(\Tilde{x}z)$ for any $z\in \mathcal{H}$ and $\Tilde{x}\in SL_2(\mathbb{Z})$ is a pre-image of $x\in SL_2(\mathbb{F})$. The map is independent of the choice of $\Tilde{x}$ since for any $x_1,x_2$ be pre-images of $x$ then $x_1x_2^{-1}\in \Gamma(q)$ and $$j(x_1,z)^{-k}f(x_1z)=j(x_1x_2^{-1},x_2z)^{-k}j(x_2,z)^{-k}f(x_1x_2^{-1}x_2z)=j(x_2,z)^{-k}f(x_2z).$$ So, we will use $x$ instead of $\Tilde{x}$. To show the map is well-defined we need to show that $\mathcal{F}f(z)\in\rho(\alpha)$ and $\mathcal{F}(f)$ satisfies the automorphy condition. For $b\in B$, $x\in SL_2(\mathbb{F})$, $z\in\mathcal{H}$ and $g\in SL_2(\mathbb{Z})$,
\begin{align*}
    [\mathcal{F}f(z)](bx) & =j(bx,z)^{-k}f(bxz)=j(bx,z)^{-k}\alpha(b)j(b,xz)^kf(xz)=\alpha(b)j(x,z)^{-k}f(xz)=\alpha(b)[\mathcal{F}f(z)](x),\\ [\mathcal{F}f(gz)](x) & =j(x,gz)^{-k}f(xgz)=j(g,z)^kj(xg,z)^{-k}f(xgz) \\ & =j(g,z)^k[\mathcal{F}f(z)](xg)=j(x,z)^k[\rho(\alpha)(g)\cdot \mathcal{F}f(z)](x).
\end{align*}


Let $\mathcal{G}:\mathcal{S}_k(SL_2(\mathbb{Z}),\rho(\alpha))\rightarrow \mathcal{S}_k(\Gamma_0(q),\alpha)$ be a linear map defined as $\mathcal{G}f(z):=[f(z)](1)$. The map is well-defined since for $b\in \Gamma_0(q)$, $$\mathcal{G}f(bz)=f(bz)(1)=j(b,z)^k[\rho(\alpha)(b)\cdot f(z)](1)=j(b,z)^k[f(z)](b)=j(b,z)^k\alpha(b)[f(z)](1)=j(b,z)^k\alpha(b)\mathcal{G}f(z).$$

To complete a proof of Theorem \ref{theo1}, we need to show that $\mathcal{G}$ is an inverse of $\mathcal{F}$, which is easy to see. This proof is an effect of Shapiro's lemma on parabolic group cohomology, isomorphic to the space of cusp forms via the Eichler-Shimura isomorphism. For the Steinberg representation, let $\alpha={\bf{1}}$ then $\mathcal{S}_k(SL_2(\mathbb{Z}),\rho({\bf{1}}))=\mathcal{S}_k(SL_2(\mathbb{Z}),{\bf{St}})\oplus \mathcal{S}_k(SL_2(\mathbb{Z}))$, and so $$\dim \mathcal{S}_k(SL_2(\mathbb{Z}),{\bf{St}})= \dim \mathcal{S}_k(\Gamma_0(q)) - \dim \mathcal{S}_k(SL_2(\mathbb{Z})).$$

\subsection{Quadratic Case}
For the quadratic character $\zeta_e$ on $\mathbb{{F}}^*$, the parabolically induced representation $\rho(\zeta_e)$ is a direct sum of two irreducible subrepresentation $\rho(\zeta_e)^{\pm}$. So, $$\mathcal{S}_k(SL_2(\mathbb{Z}),\rho(\zeta_e))=\mathcal{S}_k(SL_2(\mathbb{Z}),\rho(\zeta_e)^+)\oplus \mathcal{S}_k(SL_2(\mathbb{Z}),\rho(\zeta_e)^-).$$
We will use the character formula to find the dimensions of $\mathcal{S}_k(SL_2(\mathbb{Z}),\rho(\zeta_e)^{\pm})$. By using the character tables \ref{character table}, \ref{table}, Equation \ref{eq2} and Remark \ref{rmk3},
 
\begin{align*}
   \dim \mathcal{S}_k(SL_2(\mathbb{Z}),\rho(\zeta_e)^{\pm}) = &\textbf{ } \frac{\zeta_e(1)}{2q(q-1)}\dim \mathcal{S}_k(\Gamma(q))+\frac{\zeta_e(-1)}{2q(q-1)}(-1)^k\dim \mathcal{S}_k(\Gamma(q)) \\ & + \frac{1}{q(q-1)}\sum_{x\in\mathbb{F}^*}\frac{1}{2}(1\pm\zeta_e(x)\sqrt{q\zeta_e(-1)})\sum\limits_{n\in\mathbb{F}}\psi_n(x)\dim \mathcal{S}_k(\Gamma_1(q),\psi_n) \\ & + \frac{1}{q(q-1)}\sum_{x\in\mathbb{F}^*}\zeta_e(-1)\frac{1}{2}(1\pm\zeta_e(x)\sqrt{q\zeta_e(-1)})(-1)^k\sum\limits_{n\in\mathbb{F}}\psi_n(x)\dim \mathcal{S}_k(\Gamma_1(q),\psi_n) \\ & + \frac{1}{2(q-1)}\sum_{x\neq\pm1}\zeta_e(x)\sum\limits_{\phi:\mathbb{F}^*\rightarrow\mathbb{C}^*}\phi(x)\dim \mathcal{S}_k(\Gamma_0(q),\phi)
\end{align*}

If $\zeta_e(-1)\neq (-1)^k$ then $\mathcal{S}_k(SL_2(\mathbb{Z}),\rho(\zeta_e)^{\pm})=0$. So assume, $\zeta_e(-1)=(-1)^k$. 

\begin{align*}
   \dim \mathcal{S}_k(SL_2(\mathbb{Z}),\rho(\zeta_e)^{\pm}) = &\textbf{ } \frac{1}{q(q-1)}\dim \mathcal{S}_k(\Gamma(q)) + \frac{1}{q(q-1)}\sum\limits_{n\in\mathbb{F}}\bigg(\sum_{x\in\mathbb{F}^*}\psi_n(x)\bigg)\dim \mathcal{S}_k(\Gamma_1(q),\psi_n) \\ & \pm \frac{\sqrt{q\zeta_e(-1)}}{q(q-1)}\sum\limits_{n\in\mathbb{F}}\bigg(\sum_{x\in\mathbb{F}^*}\zeta_e(x)\psi_n(x)\bigg)\dim \mathcal{S}_k(\Gamma_1(q),\psi_n) \\ & + \frac{1}{2(q-1)}\sum\limits_{\phi:\mathbb{F}^*\rightarrow\mathbb{C}^*}\bigg(\sum_{x\in\mathbb{F}^*}\zeta_e(x)\phi(x)\bigg)\dim \mathcal{S}_k(\Gamma_0(q),\phi) \\ & - \frac{1}{2(q-1)}\sum_{x=\pm1}\zeta_e(x)\sum\limits_{\phi:\mathbb{F}^*\rightarrow\mathbb{C}^*}\phi(x)\dim \mathcal{S}_k(\Gamma_0(q),\phi)
\end{align*}

For a fixed character $\phi$ on $\mathbb{F}^*$, the sum $\sum_{x\in\mathbb{F}^*}\zeta_e(x)\phi(x)$ is equal to $0$ unless $\phi=\zeta_e$. Similarly, for a fixed $n\in\mathbb{{F}}$, the sum $\sum_{x\in\mathbb{F}}\psi_n(x)$ is equal to $0$ unless $n=0$. Note that the sum $\sum_{x\in\mathbb{F}^*}\zeta_e(x)\psi_n(x)$ is the quadratic Gauss sum if $n\neq 0$, and so it is equal to $\zeta_e(n)\sqrt{q\zeta_e(-1)}$ (Chapter 6 in \cite{Ireland}). Note that we are following the usual convention that the sign of $\sqrt{x}$ is positive. We conclude the proof of Proposition \ref{prop3} by using Remark \ref{rmk6},
\begin{align*}
    \dim \mathcal{S}_k(SL_2(\mathbb{Z}),\rho(\zeta_e)^{\pm}) = \frac{1}{2}\dim \mathcal{S}_k(\Gamma_0(q),\zeta_e)\pm \frac{\zeta_e(-1)}{q-1}\sum_{n\in\mathbb{{F}^*}} \zeta_e(n)\mathcal{S}_k(\Gamma_1(q),\psi_n).
\end{align*}

\section{Cusp forms with values cuspidal representations}

We will find the dimension of $\mathcal{S}_k(SL_2(\mathbb{Z}),\omega(\tau_0))$ for a character $\tau_0:\mathbb{E}_1^*\rightarrow \mathbb{C}^*$ by using the character formula. If $\tau_0(-1)\neq (-1)^k$ then $\mathcal{S}_k(SL_2(\mathbb{Z}),\omega(\tau_0))=0$. So, we assume that $\tau_0(-1)=(-1)^k$. By using Character tables \ref{character table},\ref{table} and Remark \ref{rmk3},

\begin{align*}
   \dim \mathcal{S}_k(SL_2(\mathbb{Z}),\omega(\tau_0)) = &\textbf{ } \frac{\tau_0(1)}{q(q+1)}\dim \mathcal{S}_k(\Gamma(q))+\frac{\tau_0(-1)}{q(q+1)}(-1)^k\dim \mathcal{S}_k(\Gamma(q)) \\ & + \frac{1}{q(q-1)}\sum\limits_{x\in\mathbb{{F}^*}}(-\tau_0(1))\sum\limits_{n\in\mathbb{F}}\psi_n(x)\dim \mathcal{S}_k(\Gamma_1(q),\psi_n) \\ & + \frac{1}{q(q-1)}\sum\limits_{x\in\mathbb{{F}^*}}(-\tau_0(-1))(-1)^k\sum\limits_{n\in\mathbb{F}}\psi_n(x)\dim \mathcal{S}_k(\Gamma_1(q),\psi_n) \\ & + \frac{1}{2(q+1)}\sum_{z\neq\pm1}-(\tau_0(z)+\overline{\tau_0}(z))\sum\limits_{\tau:\mathbb{E}_1^*\rightarrow\mathbb{C}^*}\tau(z)\dim \mathcal{S}_k(\Gamma_{\mathbb{E}}(q),\tau) \\
    \dim \mathcal{S}_k(SL_2(\mathbb{Z}),\omega(\tau_0)) = &\textbf{ } \frac{2}{q(q+1)}\dim \mathcal{S}_k(\Gamma(q))+ \frac{1}{2(q+1)}\sum_{z=\pm 1}2\tau_0(z)\sum\limits_{\tau:\mathbb{E}_1^*\rightarrow\mathbb{C}^*}\tau(z)\dim \mathcal{S}_k(\Gamma_{\mathbb{E}}(q),\tau) \\ & - \frac{2}{q(q-1)}\sum_{n\in \mathbb{{F}}}\bigg(\sum_{x\in\mathbb{F}^*}\psi_n(x)\bigg)\dim \mathcal{S}_k(\Gamma_1(q),\psi_n) \\ & - \frac{1}{2(q+1)}\sum_{z\in\mathbb{E}_1^*}(\tau_0(z)+\overline{\tau_0}(z))\sum\limits_{\tau:\mathbb{E}_1^*\rightarrow\mathbb{C}^*}\tau(z)\dim \mathcal{S}_k(\Gamma_{\mathbb{E}}(q),\tau).
\end{align*}

For a fixed character $\tau$ on $\mathbb{E}_1^*$, the sum $\sum_{z\in\mathbb{E}_1^*}\tau(z)\tau_0(z)$ is equal to $0$ unless $\tau=\overline{\tau_0}$. Similarly for any $n\in\mathbb{F}$, the sum $\sum_{x\in\mathbb{F}}\psi_n(x)$ is equal to $0$ unless $n=0$. We get the following equality by using Remark \ref{rmk5}.

\begin{align*}
    \dim \mathcal{S}_k(SL_2(\mathbb{Z}),\omega(\tau_0)) =  \textbf{ } \frac{2}{q-1}\dim \mathcal{S}_k(\Gamma(q))-\frac{2}{q-1}\dim \mathcal{S}_k(\Gamma_1(q)) - \frac{1}{2}\dim \mathcal{S}_k(\Gamma_{\mathbb{E}}(q),\overline{\tau_0}) - \frac{1}{2}\dim \mathcal{S}_k(\Gamma_{\mathbb{E}}(q),\tau_0).
\end{align*}

We prove Theorem \ref{theo4} in the case of non-quadratic character on $\mathbb{E}_1^*$ by using the following explicit formula of dimensions of $\mathcal{S}_k(\Gamma(q))$ and $\mathcal{S}_k(\Gamma_1(q))$ (Figure 3.4 in \cite{diamond}). For $k\geq3$, 
\begin{align*}
    \dim \mathcal{S}_k(\Gamma_1(q)) = & \textbf{ }\frac{k-1}{24}(q^2-1)-\frac{1}{2}(q-1)\\ 
    \dim \mathcal{S}_k(\Gamma(q)) = & \textbf{ }\frac{k-1}{24}q(q^2-1)-\frac{1}{4}(q^2-1).
\end{align*}

\subsection{Quadratic Case}
If $\tau_0=\zeta_0$ is a quadratic character on $\mathbb{E}_1^*$ then the special Weil representation $\omega(\zeta_0)=\omega(\zeta_0)^+\oplus \omega(\zeta_0)^-$ as irreducible representations. So, $\mathcal{S}_k(SL_2(\mathbb{Z}),\omega(\zeta_0))=\mathcal{S}_k(SL_2(\mathbb{Z}),\omega(\zeta_0)^+)\oplus\mathcal{S}_k(SL_2(\mathbb{Z}),\omega(\zeta_0)^-)$. We follow the same method to get Proposition \ref{prop5}.


\pagebreak

\bibliographystyle{unsrt}  


\end{document}